\documentclass{article}[12pt]
\usepackage{url, graphicx, xcolor, hyperref, geometry}
\usepackage{latexsym,amsbsy}
\usepackage{lipsum}
\usepackage{amsfonts}
\usepackage{graphicx}
\usepackage{epstopdf}
\usepackage{amssymb,amsmath,amsthm}
\usepackage{smartdiagram}
\usepackage{mathabx}
\usepackage{float}
 
\usepackage{enumerate}

\newtheorem{theorem}{Theorem}[section]
\newtheorem{lemma}{Lemma}[section]
\newtheorem{corollary}{Corollary}[section]
\newtheorem{remark}{Remark}[section]

\newtheorem{definition}{Definition}[section]

\hypersetup{
    colorlinks,
    linkcolor={blue},
    citecolor={blue},
    urlcolor={blue}
}

\begin{document}
 
\title{A novel multigrid method for elliptic distributed control problems}

\author{Yunhui He\thanks{Department of Computer Science, The University of British Columbia, Vancouver, BC, V6T 1Z4, Canada,   \tt{yunhui.he@ubc.ca}.}  }

\maketitle

\begin{abstract}
 Large linear systems of saddle-point type  have arisen in a wide variety of applications throughout computational science and engineering.  The discretizations of distributed control problems have a saddle-point structure.  The numerical solution of saddle-point problems has attracted considerable interest in recent years.   
 In this work, we propose a novel  Braess-Sarazin  multigrid  relaxation scheme  for finite element discretizations of the distributed control problems, where we use  the stiffness matrix   obtained from the  five-point finite difference method for the Laplacian to approximate the inverse of the mass matrix arising in the saddle-point system.   We apply local Fourier analysis to examine   the smoothing properties of the Braess-Sarazin multigrid relaxation.  From our analysis, the optimal smoothing factor for  Braess-Sarazin relaxation is derived.  Numerical experiments  validate our theoretical results. The relaxation scheme considered here shows its high efficiency and robustness with respect to the regularization parameter and grid size.
\end{abstract}


\vspace{2mm} 

\section{Introduction} \label{sec:intro}
 
 In this paper we consider the following two-dimensional elliptic optimal control
problem  \cite{rees2010optimal}: Find the state $u^{*} \in H^{1}(\Omega)$ and the control $f^{*}\in L^{2}(\Omega)$ such that
\begin{equation}\label{eq:Control-problem}
  \mathcal{J}(u^{*},f^{*}) = \displaystyle\min_{(u,f)\in H^{1}(\Omega)\times L^{2}(\Omega)}\mathcal{J}(u,f),
\end{equation}
subject to the state equation
\begin{equation}\label{eq:sub-state-equ}
\left\{
  \begin{aligned}
  -\triangle u &=&f,\,\,\,\, {\rm in}\,\,\Omega,\\
  u&=&g,\,\,\, {\rm at}\,\,\, \partial \Omega,\\
   \end{aligned}
               \right. \\
\end{equation}
with cost functional
\begin{equation}\label{eq:cost-fun}
  \mathcal{J}(u,f) = \frac{1}{2}\|u-\hat{u}\|^2_{L^2(\Omega)}+\beta\|f\|^2_{L^2(\Omega)},
\end{equation}
where $\hat{u}$ is the desired state and $\beta>0$ is the weight of the cost of the control
(or simply a regularization parameter).

The discretization of the optimization problem \eqref{eq:Control-problem} leads to a large scale saddle-point system.  Since  multigrid methods offer
the possibility of solving problems with $N$ unknowns using $O(N)$ work and storage, and lead to substantial improvements in computational efficiency over direct methods, they have gained growing interest in the area of optimal control problems, see,  for example \cite{borzi2003multigrid,borzi2007high,borzi2009multigrid, engel2011multigrid, schoberl2011robust, simon2009schwarz, takacs2011convergence}.

The choice of multigrid smoother is crucial to construct efficient multigrid algorithms.  Local Fourier analysis (LFA) is a useful tool to  help choose multigrid components, such as relaxation schemes and grid-transfer operators, and can quantitatively predict the two-grid convergence factor.  The main goal of this work is to construct and analyze  multigrid methods  for the discrete optimal control problems.  Note that there are a few studies of LFA applied to the optimal control problems.  In \cite{kepler2009fourier},  the smoothing factor for a one-dimensional optimal control problem was proposed by LFA for a special smoother. \cite{borzi2009multigrid} reviewed some recent efforts and outlined recent developments in the field of multigrid methods for partial differential equation optimization, where LFA was applied to study collective Gauss–Seidel methods. The finite difference multigrid solution of an optimal control problem associated with an elliptic equation was  considered in \cite{borzi2002accuracy}, and LFA was used to estimate the convergence factor for the collective Gauss–Seidel relaxation. 
 
 In this work,  we consider  finite element  monolithic multigrid solution of the  elliptic distributed control problem  \eqref{eq:Control-problem}, that is, employing multigrid methods directly to the discrete system in coupled form.   In contrast to collective smoothers, we consider a well-known Braess-Sarazin-type relaxation (BSR) scheme, which was original  designed for the Stokes equations \cite{braess1997efficient}.   In our recently work \cite{CH2021addVanka}, we have proved that the mass matrix  obtained from bilinear finite elements is a good approximation to the inverse of the scalar Laplacian discretized by a five-point finite difference method, where the optimal smoother factor of $\frac{1}{3}$ is derived for the Laplacian. Furthermore, we extended this mass-based approximation to a Braess-Sarazin-type relaxation scheme for the Stokes equations \cite{YH2021massStokes}, where we obtained the optimal smoother factor of $\frac{1}{3}$  for Stokes. This motivates us to consider the stiffness matrix derived from the finite-difference discrerization of  the scalar Laplacian to approximate the mass matrix arsing from the finite element discretizations of optimal control problems.  In the resulting BSR relaxation, we can directly apply the stiffness matrix and there is no need to invert a matrix. LFA is applied  to examine this stiffness-based BSR applied to the optimal control problems considered here. We derive  optimal smoothing factor of $\frac{1}{3}$ for the stiffness-based BSR for the optimal control problems. Considering practical use, an inexact version of stiffness-based Braess-Sarazin relaxation is developed, where a  3 $V(2,2)$-cycles  multigrid  with weighted Jacobi relaxation is used to approximate the solution of Schur complement system appeared in the exact BSR process.  This inexact version preserves the optimal smoothing factor of $\frac{1}{3}$ obtained from exact BSR. The LFA  convergence estimates agree very well with results of numerical experiments and   are independent of the grid size and of the value of the control parameter.

This paper is organized as follows.  In Section \ref{sec: Discretization-relaxation}, we propose a  stiffness-based Braess-Sarazin relaxation scheme for the  elliptic distributed control problem.  
In Section \ref{sec:smoothing-analysis}, we apply LFA to the proposed stiffness-based Braess-Sarazin relaxation, and we obtain a highly satisfactory convergence factor. Numerical results are reported in Section \ref{sec:Numer} to confirm our theoretical analysis. Finally, we draw conclusions in Section \ref{sec:concl}.

\section{Discretization and relaxation}\label{sec: Discretization-relaxation}

In our work, we  discretize  problem \eqref{eq:Control-problem}  using bilinear quadrilateral $Q_1$ finite elements. The discretize formulations of \eqref{eq:sub-state-equ} and \eqref{eq:cost-fun} are given by
\begin{eqnarray}
&&\min_{u_h,f_h}\frac{1}{2}\|u_h-\hat{u}\|^2_{L^2(\Omega)}+\beta\|f_h\|^2_{L^2(\Omega)},\label{eq:weak-form1}\\
&&{\rm such\,\, that} \int \nabla u_h v_h = \int v_h f_h, \,\,\forall v_h\in V_0^h, \label{eq:weak-form2}
\end{eqnarray}
where $V_0^h$ is the $Q_1$ finite element space. With a Lagrange multiplier ($\tau$) approach to  \eqref{eq:weak-form1} and  \eqref{eq:weak-form2}, it typically leads to a linear system of the form  \cite{rees2010optimal}
 \begin{equation}\label{eq:3-block-system}
\begin{pmatrix}
2 \beta M   &     0     & -M^T\\
0              & M       &K^T \\
 -M           &K       & 0
 \end{pmatrix}
\begin{pmatrix} f_h \\ u_h \\ \tau_h \end{pmatrix}
 =b_h,
 \end{equation}
 where $K$ is  the stiffness matrix (the discrete Laplacian) and  $M$ is the discrete mass matrix. For more details, we refer to \cite{rees2010optimal}.
 
We can rewrite  system  \eqref{eq:3-block-system} as a saddle-point system, that is,
 \begin{equation}\label{eq:saddle-structure}
       \mathcal{L}_{h}  z=\begin{pmatrix}
      A & B^{T}\\
     B & 0\\
    \end{pmatrix}
        \begin{pmatrix} x \\ y
        \end{pmatrix}
  =b_h,
 \end{equation}
where  
\begin{equation*}
A=\begin{pmatrix}
2 \beta M   &     0 \\
0              & M  
\end{pmatrix},\quad  
B= \begin{pmatrix}
-M  &     K
\end{pmatrix}.
\end{equation*}
%
Here, we consider multigrid methods for solving \eqref{eq:saddle-structure}. One important process in multgrid is the smoothing, where we update a current approximation $z_k$ via 
\begin{equation}\label{eq:general relaxation-form}
\hat{z}_{k} = z_k-\omega \mathcal{K}^{-1}_h (b_h-  \mathcal{L}_{h}  z_k)=z_k-\omega \delta z,
\end{equation}
where $\mathcal{K}^{-1}_h$ is an approximation to  $\mathcal{L}_{h} $, and $\delta z =\mathcal{K}^{-1}_h (b_h  - \mathcal{L}_{h} z_k)$.   Then, the error-propagation for the relaxation scheme \eqref{eq:general relaxation-form} is 
\begin{equation}\label{eq:Relaxation-error-S}
\mathcal{S}_h=I-\omega   \mathcal{K}^{-1}_h \mathcal{L}_h.
\end{equation}
The choices of $\mathcal{K}_h$ and $\omega$  play a crucial role in determining the convergence speed of multigrid methods.   For
simplicity, throughout the rest of this paper, we drop the subscript $h$, except when necessary for clarity.

In this work, we consider  Braess-Sarazin relaxation \cite{braess1997efficient}, that is,   $\mathcal{K}$ has the form of
\begin{equation}\label{eq:Precondtion}
   \mathcal{K}=  \begin{pmatrix}
    \alpha C & B^{T}\\
     B & 0\\
    \end{pmatrix},
\end{equation}
where $C$ is an approximation of $A$, whose inverse is easy to apply, and $\alpha>0$. 

With $\mathcal{K}$ defined in \eqref{eq:Precondtion}, we compute the update $\delta z=(\delta x, \delta y)$ in \eqref{eq:general relaxation-form} by  the following two stages
\begin{eqnarray}
  (BC^{-1}B^{T})\delta y&=&BC^{-1}r_{x}-\alpha r_{y}, \label{eq:solution-of-precondtion-s1}\\
  \delta x&=&\frac{1}{\alpha}C^{-1}(r_{x}-B^{T}\delta y),\label{eq:solution-of-precondtion-s2}
\end{eqnarray}
where $(r_{x},r_{y})=b_h-\mathcal{L}z_k$.

Note that, in \eqref{eq:solution-of-precondtion-s1} and   \eqref{eq:solution-of-precondtion-s2}, we can directly consider an approximation to  $C^{-1}$.   In our recently work \cite{CH2021addVanka}, we have shown that the mass matrix $M$ derived from bilinear finite elements is a good approximation to the inverse of the scalar Laplacian discretized by the five-point finite difference method. This inspires us to  consider the  discrete scalar Laplacian approximating  the inverse of the mass matrix $M$ in  the diagonal part of  the coefficient matrix defined in \eqref{eq:3-block-system}. Let $A_{fd}$ be the stiffness matrix obtained from the five-point finite difference discretization for the scalar Laplacian. In \eqref{eq:Precondtion}, we consider 
\begin{equation}\label{eq:C-inverse-Q}
C^{-1}=C^{-1}_f = 
\begin{pmatrix}
\frac{1}{2\beta}A_{fd} &  0\\
0 &  A_{fd}
\end{pmatrix},
\end{equation}
and denote the corresponding  $\mathcal{K}$ as 
\begin{equation}\label{eq:Precondtion-K1}
   \mathcal{K}_f=  \begin{pmatrix}
    \alpha C_f& B^{T}\\
     B & 0\\
    \end{pmatrix}.
\end{equation}
To choose suitable values of  the parameters $\omega$ and $\alpha$, we consider LFA to examine the relaxation scheme with $\mathcal{K}=\mathcal{K}_f$ in the following section.
%

\section{Smoothing analysis}\label{sec:smoothing-analysis} 
LFA \cite{MR1807961, wienands2004practical} has been widely used to analyze  actual multigrid performance for different types of problems, by investigating  the spectral radius of {\it symbol} of the underlying operator.  Specifically, the LFA smoothing factor  with simplifying  assumptions on boundary conditions offers a  sharp convergence estimate of multigrid methods. Thus, in this work, we employ LFA to examine the LFA smoothing factor for  BSR,  and help choose relaxation parameters.  Next, we first give a brief introduction to LFA. 

We consider two-dimensional infinite uniform grids $\mathbf{G}_{h}$
\begin{equation*}
  \mathbf{G}_{h}=\left\{\boldsymbol{x}:=(x_1,x_2)=\boldsymbol{k}h=(k_{1},k_{2})h, k\in \mathbb{Z}\right\},
\end{equation*}
and Fourier modes $\varphi(\boldsymbol{\theta},\boldsymbol{x})=e^{\imath \boldsymbol{\theta}\boldsymbol{x}/h}$ on $\mathbf{G}_{h}$, where $\boldsymbol{\theta}=(\theta_1,\theta_2)$ and $\imath^2=-1$.  Since $\varphi(\boldsymbol{\theta},\boldsymbol{x})$ is periodic in $\boldsymbol{\theta}$ with period $2\pi$, we consider $\theta_i\in\big(-\frac{\pi}{2},\frac{3\pi}{2}\big]$. The coarse grid $  \mathbf{G}_{2h}$ is defined similarly.

Let $L_h$ be a Toeplitz operator acting on $\mathbf{G}_{h}$ as follows \cite{MR1807961},
\begin{eqnarray*} 
  L_{h}w_{h}(\boldsymbol{x})&=&\sum_{\boldsymbol{\kappa}\in\boldsymbol{V}}s_{\boldsymbol{\kappa}}w_{h}(\boldsymbol{x}+\boldsymbol{\kappa}h),
  \end{eqnarray*}
with constant coefficients $s_{\boldsymbol{\kappa}}\in \mathbb{R}$, where $w_{h}(\boldsymbol{x})$ is a function in $l^{2} (\mathbf{G}_{h})$. Here, $\boldsymbol{V}$ is a finite index set. 

\begin{definition}\label{formulation-symbol}
  We call $\widetilde{L}_{h}(\boldsymbol{\theta})=\displaystyle\sum_{\boldsymbol{\kappa}\in\boldsymbol{V}}s_{\boldsymbol{\kappa}}e^{\iota \boldsymbol{\theta}\cdot \boldsymbol{\kappa}}$ the symbol of $L_{h}$.
\end{definition}

We consider multigrid methods for finite-element discretizations with standard geometric grid coarsening, that is, we construct a sequence of coarse grids by doubling the mesh size in each spatial direction.  High and low frequencies for standard coarsening are given by
\begin{equation*}
  \boldsymbol{\theta}\in T^{{\rm low}} =\left[-\frac{\pi}{2},\frac{\pi}{2}\right)^{2}, \, \boldsymbol{\theta}\in T^{{\rm high}} =\displaystyle \left[-\frac{\pi}{2},\frac{3\pi}{2}\right)^{2} \bigg\backslash \left[-\frac{\pi}{2},\frac{\pi}{2}\right)^{2}.
\end{equation*}
\begin{definition} The LFA smoothing factor for  the error-propagation for the relaxation scheme $\mathcal{S}$, see \eqref{eq:Relaxation-error-S}, is defined as
\begin{equation}\label{eq:LFA-mu-form}
  \mu_{{\rm loc}}(\boldsymbol{p})=\max_{\boldsymbol{\theta}\in T^{{
  \rm high}}}\big\{\big|\lambda(\widetilde{\mathcal{S}}(\boldsymbol{\theta}))\big| \,\,\big\},
\end{equation}
where  $\boldsymbol{p}$ is algorithmic parameters,  and $\lambda\big(\widetilde{\mathcal{S}}(\boldsymbol{\theta})\big)$ denotes the  eigenvalue of symbol $\widetilde{\mathcal{S}}(\boldsymbol{\theta})$. 
\end{definition}
Note that for the $Q_1$  elements considered in this work,  $\widetilde{\mathcal{S}}(\boldsymbol{\theta})$ is a $3\times 3$ matrix. Since $\mu_{{\rm loc}}(\boldsymbol{p})$ is a function of   $\boldsymbol{p}$, we can minimize  $\mu_{{\rm loc}}(\boldsymbol{p})$ to obtain a fast convergence speed.  Thus, we define the optimal smoothing factor as follows.
\begin{definition}
The optimal LFA  smoothing factor  for  the error-propagation for the relaxation scheme $\mathcal{S}$ is defined as
  \begin{equation*}
    \mu_{{\rm opt}}=\min_{\boldsymbol{p} \in \Upsilon}{\mu_{{\rm loc}}}(\boldsymbol{p}),
  \end{equation*}
  where $\Upsilon$ is the set of allowable parameters. 
\end{definition}

Another important LFA factor is the LFA two-grid convergence factor.  In general, the two-grid error-propagation can be expressed as
\begin{equation*}
E_h = \mathcal{S}^{\nu_2} (I-P (\mathcal{L}_{2h})^{-1} R \mathcal{L}_h)\mathcal{S}^{\nu_1},
\end{equation*} 
where $\mathcal{L}_{2h}$ is the coarse grid operator,  $R$ and $P$ are the restriction and interpolation operators, respectively.  The integers $\nu_1$ and $\nu_2$ are pre- and post- smoothing steps, respectively.  

\begin{definition}
The   LFA  two-grid convergence factor for $E_h$ is defined as
  \begin{equation}\label{eq:LFA-rho-form}
    \rho=\max_{\boldsymbol{\theta}\in T^{{
  \rm low}}} \left\{\rho(\widetilde{E}_h(\boldsymbol{\theta}))\right\} ,
  \end{equation}
  where $\widetilde{E}_h(\boldsymbol{\theta})$ is  the two-grid LFA symbol of  $E_h$, and $\rho(\widetilde{E}_h)$ denotes the spectral radius of matrix $\widetilde{E}_h$.
\end{definition}
Note that in our case,  $\widetilde{E}_h(\boldsymbol{\theta})$  is a $12\times 12$ matrix. For more details on how to compute two-grid symbol, we refer to  \cite{MR1807961}. Since it is easy to compute the LFA smoothing factor, \eqref{eq:LFA-mu-form}, compared to the LFA two-grid convergence factor, \eqref{eq:LFA-rho-form}, we focus on the analysis of LFA smoothing factor in  the following.
 \subsection{Fourier representation of mass and stiffness operators}
In one dimension (1D), the  stiffness and mass stencils of $Q_1$ elements are
\begin{equation}\label{eq:1D-Q1Q1-stencil}
K_1=\frac{1}{h}\begin{bmatrix}  -1 & 2 &-1  \end{bmatrix} \quad \text{and} \quad 
M_1=\frac{h}{6}\begin{bmatrix}  1 & 4 &1  \end{bmatrix},
\end{equation}
respectively.

Then, the stiffness stencil for bilinear discretization in two dimensions (2D)  is given by
\begin{eqnarray*}
  K &=& K_1\otimes M_1+M_1\otimes K_1 \nonumber\\
   &=& \frac{1}{3}\begin{bmatrix}
   -1 & -1 &-1\\
   -1 & 8 &-1\\
   -1 & -1 &-1
   \end{bmatrix}, 
\end{eqnarray*}
and the mass stencil for bilinear discretization  in 2D is given by
\begin{eqnarray}
  M &=& M_1\otimes M_1 \nonumber\\
   &=& \frac{h^2}{36}\begin{bmatrix}
   1 &  4    & 1\\
   4 &  16   & 4\\
   1 &  4    & 1
   \end{bmatrix}.\label{eq:mass-2D-stencil}
\end{eqnarray}
From  \eqref{eq:1D-Q1Q1-stencil}, the symbols of the stiffness and mass stencils in 1D are
\begin{equation}\label{eq:1D-stencil-symbol}
\widetilde{K}_1(\theta)=\frac{2}{h}(1-\cos\theta),\,\,\,
\widetilde{M}_1(\theta)=\frac{h}{3}(2+\cos\theta).
\end{equation}
Based on \eqref{eq:1D-stencil-symbol}, the symbol of $K$ can be written as
\begin{eqnarray*}
  \widetilde{K}(\theta_1,\theta_2) &=&\widetilde{K}_1(\theta_2)\widetilde{M}_1(\theta_1)+
  \widetilde{M}_1(\theta_2) \widetilde{K}_1(\theta_1) \nonumber \\
  &=&\frac{2}{3}(4-\cos\theta_1-\cos\theta_2-2\cos\theta_1\cos\theta_2),
\end{eqnarray*}
and the symbol of $M$ is
\begin{eqnarray*}
  \widetilde{M}(\theta_1,\theta_2) &=&\widetilde{M}_1(\theta_2) \widetilde{M}_1(\theta_1)\nonumber \\
  &=&\frac{h^2}{9}(4+2\cos\theta_1+2\cos\theta_2+\cos\theta_1\cos\theta_2).
\end{eqnarray*}
We introduce  the stencil of the Laplacian $-\triangle_{h}$ discretized by the standard five-point scheme, given by
\begin{equation*}
 A_{fd}=  -\triangle_{h} =\frac{1}{h^2}\begin{bmatrix}
       & -1 &  \\
      -1 & 4 & -1 \\
        & -1 &
    \end{bmatrix}.
 \end{equation*}
Then,  the symbol of $A_{fd}$ is
\begin{equation}\label{eq:scalar-symbol-Laplace}
\widetilde{A}_{fd} =\frac{4-2\cos\theta_{1}-2\cos\theta_2}{h^2}.
\end{equation}
For simplicity, let $a= \widetilde{M}$ and $\hat{a}=(\widetilde{A}_{fd})^{-1}$.
Using \eqref{eq:scalar-symbol-Laplace}   gives
\begin{equation} \label{eq:AM-MA-symbol}
\widetilde{A}_{fd} \widetilde{M} =\widetilde{M} \widetilde{A}_{fd}= \frac{a}{\hat{a}}.
\end{equation}

\subsection{LFA for Braess-Sarazin relaxation scheme}\label{subsec:BSR-analytical-result}
 
Now, we examine  Braess-Sarazin-type algorithms, which  were originally developed as a relaxation scheme for the Stokes equations \cite{braess1997efficient}.   In practice,  \eqref{eq:solution-of-precondtion-s1} is not solved exactly. Thus, in the following, we first consider  exact solve for  \eqref{eq:solution-of-precondtion-s1}, and then inexact one.
%
%
 
 {\bf Exact Braess-Sarazin relaxation}: In \eqref{eq:Precondtion},   we consider $C^{-1}$  defined in \eqref{eq:C-inverse-Q} and $\alpha=1$. We first derive the corresponding optimal smoothing factor for the exact Braess-Sarazin relaxation.   We  refer to the relaxation scheme as  stiffness-based BSR. 
Let $b= \widetilde{K}(\theta_1,\theta_2) $.
The symbol of operator $\mathcal{L}$ defined in \eqref{eq:saddle-structure} is given by  
\begin{equation*} 
  \widetilde{\mathcal{L}}(\theta_1,\theta_2) = \begin{pmatrix}
       2\beta a&       0 &          -a  \\
                0 &      a &           b \\
              -a  &       b&           0
    \end{pmatrix},
\end{equation*}
and 
\begin{equation*} 
  \widetilde{\mathcal{K}}_f(\theta_1,\theta_2) =\begin{pmatrix}
       2\beta \hat{a}&       0 &          -a  \\
                0 &       \hat{a} &           b \\
              -a  &       b&           0
    \end{pmatrix}.
\end{equation*}
Then,
\begin{equation*}
  \widetilde{\mathcal{L}}-  \lambda \widetilde{\mathcal{K}}_f
   = \begin{pmatrix}
   2\beta (a-\lambda\hat{a})&              0                     &          -a(1-\lambda)  \\
                0 &                      a-\lambda \hat{a}         &           b(1-\lambda) \\
 -a(1-\lambda)  &                   b(1-\lambda)                            &           0
    \end{pmatrix}.
\end{equation*}
The determinant of $\widetilde{\mathcal{L}}-  \lambda \widetilde{\mathcal{K}}_f$ is 

\begin{align*}
 |\widetilde{\mathcal{L}}-  \lambda \widetilde{\mathcal{K}}_f | &= -2\beta (a-\lambda\hat{a}) b^2(1-\lambda)^2-a^2(1-\lambda)^2(a-\lambda \hat{a})\\
  & =(1-\lambda)^2 (a-\lambda \hat{a})(-2\beta b^2-a^2).
\end{align*}
It follows that the eigenvalues of   $\widetilde{\mathcal{K}}^{-1}_f\widetilde{\mathcal{L}}$ are $1,1$ and $\frac{a}{\hat{a}}$.

We restate the results for  solving the Laplace problem using  mass matrix approximation \cite{CH2021addVanka} in the following.

\begin{lemma} \cite{CH2021addVanka}\label{lem:ratio-a-hat-a}
For $\boldsymbol{\theta} \in T^{high}$, $\frac{a}{\hat{a}}(\boldsymbol{\theta}) \in [8/9, 16/9]$.
\end{lemma}
\begin{lemma} \cite{CH2021addVanka}\label{lemma:mass-smoothing-factor}
If we consider the mass matrix $M$ to approximate the inverse of the scalar  Laplacian, $A_{fd}$, and  $\mathcal{S}=I-\omega MA_{fd}$, then  the optimal smoothing factor for $\mathcal{S}$  is
\begin{equation*} 
  \mu_{\rm opt}= \min_{\omega} \max_{\boldsymbol{\theta}\in T^{{
  \rm high}}} \left|1- \omega \widetilde{ M} \widetilde{A}_{fd}\right|=\frac{1}{3},
\end{equation*}
where the minimum is uniquely achieved at $\omega=\omega_{\rm opt} = \frac{3}{4}$. 
\end{lemma}

Since $\widetilde{ M}\widetilde{A}_{fd}=\widetilde{A}_{fd}\widetilde{ M}=\frac{a}{\hat{a}}$, see \eqref{eq:AM-MA-symbol}, we have the following result.
\begin{corollary}\label{corollary:laplace-approximate-mass}
If we consider the discrete Laplacian $A_{fd}$ to approximate the inverse of the mass matrix, $M$, and  $\mathcal{S}_{fd}=I-\omega A_{fd} M$, then  the optimal smoothing factor for $\mathcal{S}_{fd}$  is
\begin{equation*} 
  \mu_{\rm opt}= \min_{\omega} \max_{\boldsymbol{\theta}\in T^{{
  \rm high}}}  \left|1- \omega \widetilde{A}_{fd}\widetilde{ M}\right| =\frac{1}{3},
\end{equation*}
provided that $\omega=\omega_{\rm opt} = \frac{3}{4}$. 
\end{corollary}
Since  $\widetilde{\mathcal{K}}^{-1}_f\widetilde{\mathcal{L}}$ contains eigenvalue $\frac{a}{\hat{a}}$, which is an eigenvalue of $\widetilde{A}_{fd}\widetilde{ M}$, Corollary \ref{corollary:laplace-approximate-mass}  indicates that  $\frac{1}{3}$ is a lower bound on the optimal smoothing factor for the  stiffness-based  BSR.  
\begin{theorem}\label{thm:EBSR-smoothing-theorem}
The optimal smoothing factor for  the stiffness-based BSR for the optimal control system \eqref{eq:saddle-structure} is
  \begin{equation*}
   \mu_{{\rm opt},B}=\displaystyle \min_{\omega_{B}}\max_{ \boldsymbol{\theta}\in T^{{\rm high}}} \left|\lambda( \widetilde{\mathcal{S}}_{B}( \omega_{B},\boldsymbol{\theta}))\right|=\frac{1}{3},
\end{equation*}
where the minimum is attained provided that $\omega_{B}=\omega_{B,\rm opt}=\frac{3}{4}$.
\end{theorem}
\begin{proof}
 Since the eigenvalues of $\widetilde{\mathcal{S}}_{B}=I-\omega_B\widetilde{\mathcal{K}}^{-1}_f\widetilde{\mathcal{L}}$ are $1-\omega_B, 1-\omega_B$ and $1-\omega_B\frac{a}{\hat{a}}$ and Lemma \ref{lem:ratio-a-hat-a}, the smoothing factor for $\mathcal{\widetilde{S}}_{B}$ is
 \begin{equation*}
\mu_{{\rm loc}} = \max \left\{|1-\omega_B|, \left|1-\frac{16}{9}\omega_B\right|, \left|1-\frac{8}{9}\omega_B\right|\right \},
\end{equation*}
From Corollary \ref{corollary:laplace-approximate-mass}, we know that 
 \begin{equation*}
\min_{\omega_B}\max \left\{\left |1-\frac{16}{9}\omega_B\right|, \left|1-\frac{8}{9}\omega_B\right|\right \}=\frac{1}{3},
\end{equation*}
provided that $\omega_B=\omega_{B,\rm opt}=\frac{3}{4}$.  Furthermore,    $|1-\omega_{B,\rm opt}|=\frac{1}{4}<\frac{1}{3}$.   Thus,  $\min_{\omega_B}\mu_{{\rm loc}} =\frac{1}{3}$.
\end{proof}
 \begin{remark}
 We point out that one can consider $\alpha$ as a variable in \eqref{eq:Precondtion-K1} rather than $\alpha=1$. However,  this does not change the optimal smoothing factor  for exact BSR obtained in Theorem \ref{thm:EBSR-smoothing-theorem}, and only changes the choice of optimal parameter $\omega_B$. In contrast to exact BSR,  varying $\alpha$  is needed for inexact version of BSR, which will be discussed in Section \ref{sec:Numer}. 
 \end{remark}
 
 {\bf Inexact Braess-Sarazin  relaxation}:  Exact  Braess-Sarazin algorithm requires an exact inversion of the Schur complement \eqref{eq:solution-of-precondtion-s1}, which is very expensive.  In the literature, it has been shown that   inexact  solve  using a few sweeps of weighted Jacobi iteration for the Schur complement system is enough to maintain the same convergence of exact BSR for the Stokes equations, see, for example \cite{YH2021massStokes}.   In \eqref{eq:solution-of-precondtion-s1}, the Schur complement matrix is
 \begin{align*}
 BC^{-1} B^T &= \begin{pmatrix}
-M & K
\end{pmatrix}
\begin{pmatrix}
\frac{1}{2\beta}A_{fd} &  0\\
0 &  A_{fd}
\end{pmatrix}
 \begin{pmatrix}
-M^T \\   K^T
\end{pmatrix} \\
& = \frac{1}{2\beta} MA_{fd}M^T+KA_{fd}K^T.
\end{align*}
We explore  inexact stiffness-based BSR, using either a few sweeps of weighted Jacobi iteration to approximate the solution of \eqref{eq:solution-of-precondtion-s1} or a $V$-cycles multigrid on the Schur complement system. We are wondering whether the inexact version is possible to achieve the same smoothing factor of $\frac{1}{3}$. Here, we do not  seek theoretical analysis. Instead, we numerically examine  the performance of  inexact BSR in Section \ref{sec:Numer}.   

Note that another choice for $\mathcal{K}$ is   
\begin{equation}\label{eq:KJ-form}
\mathcal{K}=\mathcal{K}_D=\begin{pmatrix}
C_D & B^T\\
B  &  0
\end{pmatrix},
\end{equation}
where  
\begin{equation*}
C_D=\begin{pmatrix}
2\beta {\rm diag}(M)  & 0\\
0  & {\rm diag}(M) 
\end{pmatrix}.
\end{equation*} 
From \eqref{eq:mass-2D-stencil}, it can be seen that the symbol of ${\rm diag}(M)$ is $\frac{4h^2}{9}=:a_D$.
Here, we are curious about the optimal smoothing factor for the smoother  \eqref{eq:KJ-form}. As a comparison, we derive the optimal smoothing factor for this choice.

\begin{theorem}\label{thm:Jacobi-BSR-smoothing}
If we consider $\mathcal{K}_D$ defined in \eqref{eq:KJ-form} and $\mathcal{S}_D=I-\omega_D \mathcal{K}^{-1}_D\mathcal{L}$,  them  the optimal smoothing factor for $\mathcal{S}_D$ is 
 \begin{equation*}
   \mu_{{\rm opt},D}=\min_{\omega_{D}}\max_{ \boldsymbol{\theta}\in T^{{\rm high}}} \left|\lambda(\widetilde{\mathcal{S}}_{D}( \omega_{D},\boldsymbol{\theta}))\right|=\frac{5}{7},
\end{equation*}
provided that $\omega_D =\frac{8}{7}$.
\end{theorem}

\begin{proof}
It can easily be shown that the eigenvalues of $\widetilde{\mathcal{S}}_D$ are $1,1$ and $\frac{a}{a_D}$. Now, we consider the optimal smoothing factor for the modes, $\frac{a}{a_D}$. Note that

\begin{equation*}
\frac{a}{a_D} = \frac{1}{4}(4+2\cos\theta_1+2\cos\theta_2+\cos\theta_1\cos\theta_2).
\end{equation*}
For $\boldsymbol{\theta}\in T^{\rm high}$, it can easily be shown that the extreme values of $\frac{a}{a_D}$ are $\frac{1}{4}$ and $\frac{3}{2}$. Thus, 

 \begin{equation*}
\mu_{\rm opt}=\min_{\omega_D}\max_{ \boldsymbol{\theta}\in T^{{\rm high}}} \left\{  |1- \omega_D|,\left |1-\frac{1 }{4}\omega_D\right|, \left|1-\frac{3}{2}\omega_D\right|\right \}=\frac{5}{7},
\end{equation*}
provided that $\omega_D=\frac{2}{1/4+3/2}=\frac{8}{7}$.
\end{proof}

From Theorems \ref{thm:EBSR-smoothing-theorem} and   \ref{thm:Jacobi-BSR-smoothing}, we see that the stiffness-based BSR offers a highly satisfactory convergence speed.  Thus, in our numerical tests, we only consider the stiffness-based BSR.

\section{Numerical experiments}\label{sec:Numer}
In this section, we  first present the LFA two-grid convergence to compare with the optimal LFA smoothing factor derived in Subsection \ref{subsec:BSR-analytical-result} for the exact stiffness-based BSR.  Second,  we display actual multigrid performance to validate our analytical result for exact BSR. To consider practical use, we  develop an inexact version of BSR, where a $V$-cycle multigrid is applied to the Schur complement system to maintain the performance observed for exact BSR.  To demonstrate the robustness of our multigrid methods with respect to the regularization parameters, we consider different values of $\beta$ in our tests.

\subsection{LFA predictions}

In this subsection, we report the LFA predicted two-grid convergence factor, $\rho_h$, for different values of regularization parameter $\beta$ and grid size $h$ for exact stiffness-based BSR.   For the prolongation of correction,  we consider the standard bilinear  interpolation and apply the corresponding adjoint operator for the restriction.  Galerkin operator is used to define the coarse grid operator. The number of smoothing steps $\nu$ is   $\nu = \nu_1+\nu_2$. We use $32 \times 32$ evenly distributed Fourier frequencies in the Fourier domain $[-\frac{\pi}{2}+\tau, \frac{\pi}{2}-\tau]^2$ with $\tau=\frac{\pi}{64}$ to compute LFA two-grid convergence factor  $\rho_h$ in \eqref{eq:LFA-rho-form}.

In Table \ref{tab:LFA-beta-rho}, we show the LFA predicted two-grid convergence factor as a function of $\beta$ and $h$ with $\nu=1$. We see that the LFA  two-grid convergence factor is the same as the optimal smoothing factor derived from our theoretical analysis. Moreover,  the convergence factor is robust in  regularization parameter $\beta$ and is $h$-independent.

 \begin{table}[H]
 \caption{LFA predicted two-grid convergence factor, $\rho_h$,  as a  function of $\beta$ and $h$ with $\nu=1$ for exact stiffness-based BSR.}
\centering
\begin{tabular}{ l    c c c c}
\hline
$h$       & $\beta=10^{-2}$  &$\beta=10^{-4}$    &$\beta=10^{-6}$   & $\beta=10^{-8}$  \\
\hline
 $1/64$                    &0.333        &0.333       &0.333      &0.333   \\
 $1/128$                  &0.333        &0.333       &0.333      &0.333  \\
 $1/256$                   &0.333        &0.333       &0.333      &0.333    \\
\hline
\end{tabular}\label{tab:LFA-beta-rho}
\end{table}
In Table \ref{tab:rho-nu}, we report the LFA predicted two-grid convergence factor as a function of the number of smoothing steps $\nu$  and $h$ with $\beta=10^{-2}$. We see that $\rho_h=\mu^{\nu}_{\rm opt}$ for different $h$.  For other $\beta$, we see same behaviour,  but omit the results. 
 \begin{table}[H]
 \caption{LFA predicted two-grid convergence factor, $\rho_h$,  as a  function of   $\nu$  and $h$  with $\beta=10^{-2}$ for exact stiffness-based BSR.}
\centering
\begin{tabular}{ l  c c c }
\hline
$h$                   & $\nu =2$    & $\nu=3$          &$\nu=4$      \\
\hline
$1/64$               &0.111         &0.037               &0.012         \\
$1/128$             &0.111         &0.037               &0.012         \\
$1/256$             &0.111         &0.037               &0.012        \\
\hline
\end{tabular}\label{tab:rho-nu}
\end{table}

\subsection{Multigrid performance}
In this subsection, we report some numerical results to confirm our theoretical results. We consider the distributed control problem on the unit square domain with homogeneous Dirichlet boundary conditions and $\hat{u}=0$ such that the discrete solution is  $z_h=0$.   The coarsest grid is a $4\times 4$ mesh.   For the stiffness matrix $A_{fd}$, we consider rediscretization operator on the coarse meshgrid.   We measure the multigrid convergence factor by the following form \cite{MR1807961}
\begin{equation*}
\hat{\rho}_{h} = \sqrt[n]{\frac{||d_h^{(n)}||_2}{||d_h^{(0)}||_2}}, 
\end{equation*}
where $d_h^{(n)}=b_h-\mathcal{L}_h z_h^{(n)}$, and $z_h^{(n)}$ is the approximation to the solution of \eqref{eq:saddle-structure}  at the
$n^{\rm th}$ multigrid iteration. We display the convergence factors obtained after 100 multigrid
cycles for different values of  $\nu$.   The initial guess
$z_h^{(0)}$ is chosen randomly.   We use the MATLAB code of Rees \cite{Rees2010github} to
generate the linear systems.

Table \ref{tab:BSR-dirichlet-beta4} shows the measured $W$-cycles multigrid convergence factors for exact  stiffness-based BSR  with algorithmic parameters,  $ \alpha_B=1, \omega_B=\frac{3}{4}$, for different values of $\beta$ and grid size $h$.  We see that our numerical results are independent of grid size $h$, and   $\hat{\rho}_h$ for $\nu>1$  matches well with the LFA two-grid convergence factor  $\rho_h$. Although for $\nu>1$ the measured convergence factors are  slightly better than the predictions, it is not surprising since LFA two-grid prediction is the same as the actual convergence for problems with periodic boundary conditions, and here we consider Dirichlet boundary conditions.  
 

 \begin{table}[H]
 \caption{ $W$-cycles measured multigrid convergence factors, $\hat{\rho}_h$, for exact stiffness-based BSR with algorithmic parameters, $ \alpha_B=1, \omega_B=\frac{3}{4}$, compared with LFA two-grid predictions, $\rho_h$, for exact BSR with $ \alpha_B=1, \omega_B=\frac{3}{4}$.}
\centering
\begin{tabular}{ l  c c c c }
\hline
$\nu$  & 1    & 2   &3    & 4  \\
\hline
$\rho_h$            &0.333      &0.111         &0.037               &0.012       \\
\hline
\hline
    \multicolumn{5} {c}  {$\beta=10^{-2}$ } \\
\hline
$\hat{\rho}_{h=1/64}$              &0.279       &0.093         &0.031       &0.010         \\
$\hat{\rho}_{h=1/128}$            &0.275        &0.091         &0.030       &0.010           \\
$\hat{\rho}_{ h=1/256}$           &0.272       &0.090         & 0.030      & 0.010         \\
\hline
\hline
    \multicolumn{5} {c}  { $\beta=10^{-4}$ } \\
\hline 
$\hat{\rho}_{h=1/64}$              &0.268       &0.090         &0.030       &0.010          \\       
$\hat{\rho}_{h=1/128}$            &0.264       &0.089          &0.029        &0.010          \\   
 $\hat{\rho}_{h=1/256}$            &0.261       &0.088         &0.029       &0.010          \\   
\hline
\hline
  \multicolumn{5} {c}  { $\beta=10^{-6}$ } \\
\hline
$\hat{\rho}_{h=1/64}$              &0.289        & 0.094        &0.031       & 0.012        \\
$\hat{\rho}_{h=1/128}$            &0.276        &0.093         &0.030       &0.010          \\
$\hat{\rho}_{ h=1/256}$           &0.272       &0.092         &0.030       & 0.010         \\
\hline
\hline
  \multicolumn{5} {c}  { $\beta=10^{-8}$ } \\
\hline
$\hat{\rho}_{h=1/64}$              &0.302       & 0.098        &0.032       &0.011          \\
$\hat{\rho}_{h=1/128}$            &0.299       &0.098        &0.032       & 0.012         \\
$\hat{\rho}_{ h=1/256}$           &0.290       &0.096         &0.030       &0.012          \\
\hline
\end{tabular}\label{tab:BSR-dirichlet-beta4}
\end{table}

As mentioned before, using exact BSR, we need to solve the Schur complement system \eqref{eq:solution-of-precondtion-s1} exactly. In practice, an inexact solve is preferred. 
First, we tried two or three sweeps of weighted Jacobi iteration and Gauss-Seidel iteration to approximate the solution of  \eqref{eq:solution-of-precondtion-s1}.  However, experiments showed that these approaches maintain the performance observed for exact BSR only for $\nu=1$, and there is a degradation for $\nu>1$. Moreover, these approaches are very sensitive to algorithmic parameters.   Instead, we consider solving the Schur
complement system by applying a multigrid $V(2,2 )$-cycle using weighted Jacobi relaxation with weight $\omega_J$. In this test, we  use  $\frac{\omega_B}{\alpha_B}=\frac{3}{4}$ because $1-\frac{\omega_B}{\alpha_B}\frac{a}{\hat{a}}$ is an eigenvalue of the relaxation error-propagation operator for both exact and inexact version. We  found that  there are many choices of $\alpha_B, \omega_B=\frac{3}{4}\alpha_B$ and $ \omega_J$  resulting in good multigrid results.  In our report, we consider $\alpha_B=1.5, \omega_B=\frac{3}{4}\alpha_B$ and $ \omega_J=0.8$.    We observe that using only  2 or 3 $V(2,2)$-cycles on the approximate Schur complement achieves convergence factors essentially matching those in Table \ref{tab:BSR-dirichlet-beta4} for $\nu=1,2$.  We display the inexact BSR results  in Table \ref{tab:IBSR-dirichlet-beta0}, where 3 $V(2,2)$-cycles are used. We see there is a slight degradation  of the measured convergence factors for inexact BSR for $\nu=3,4$ compared with the results in Table \ref{tab:BSR-dirichlet-beta4}. However, for $\nu=1,2$, the measured convergence factors $\hat{\rho}$ agree very well with the results in Table \ref{tab:BSR-dirichlet-beta4}, and are $h$-independent and robust to $\beta$. In practice, we recommand $W(1,0)$-cycle and $W(1,1)$-cycle for inexact BSR.
 

 \begin{table}[H]
 \caption{$W$-cycles measured convergence factors, $\hat{\rho}_h$,  for inexact stiffness-based BSR with inner 3 $V(2,2)$-cycles and algorithmic parameters, $\alpha_B=1.5, \omega_B=\frac{3}{4}\alpha_B, \omega_J=0.8$.}
\centering
\begin{tabular}{ l  c c c c }
\hline
$\nu$  & 1    & 2   &3    & 4  \\
\hline
\hline
    \multicolumn{5} {c}  {$\beta=10^{-2}$ } \\
\hline
 
$\hat{\rho}_{h=1/64}$              &0.280        &0.092         &0.046          &  0.044    \\  
$\hat{\rho}_{h=1/128}$            &0.275         &0.092         &0.046         & 0.044          \\
$\hat{\rho}_{ h=1/256}$          &0.272         &0.090          &0.047        &0.045       \\
\hline
\hline
    \multicolumn{5} {c}  { $\beta=10^{-4}$ } \\
\hline 
$\hat{\rho}_{h=1/64}$              &0.271         &0.089         &0.045        &  0.043          \\     
$\hat{\rho}_{h=1/128}$            &0.267         &0.088         &0.046         & 0.044          \\ 
 $\hat{\rho}_{h=1/256}$           &0.264         &0.086         &0.047        &0.045          \\ 
\hline
\hline
  \multicolumn{5} {c}  { $\beta=10^{-6}$ } \\
\hline
$\hat{\rho}_{h=1/64}$              &0.283         &0.092          &0.039        &  0.025          \\
$\hat{\rho}_{h=1/128}$            &0.280       &0.091            &0.044        &  0.041           \\
$\hat{\rho}_{ h=1/256}$           &0.276        &0.090           &0.047        &0.045          \\
\hline
\hline
  \multicolumn{5} {c}  { $\beta=10^{-8}$ } \\
\hline
$\hat{\rho}_{h=1/64}$              & 0.296        &0.096         &0.032        &  0.010          \\
$\hat{\rho}_{h=1/128}$           &0.293          &0.095         &0.034         &0.025           \\
$\hat{\rho}_{ h=1/256}$           &0.289         &0.094         &0.043        &0.030          \\
\hline
\end{tabular}\label{tab:IBSR-dirichlet-beta0}
\end{table}
 Next, we consider the model problem \eqref{eq:Control-problem}  subject to the state equation \cite{rees2010optimal}
\begin{equation}\label{eq:sub-state-equ-nonzero}
\left\{
  \begin{aligned}
  -\triangle u &=&f,\,\,\,\, {\rm in}\,\,\Omega,\\
  u&=&\hat{u}|_{\partial \Omega},\,\,\, {\rm at}\,\,\, \partial \Omega,\\
   \end{aligned}
               \right.  \\
\end{equation}
where $\Omega=[0,1]^2$,  and 
 \begin{equation*} 
\hat{u} = \left\{
  \begin{aligned}
  & (2x-1)^2(2y-1)^2 & \quad   \text{if}\quad (x,y)\in[0,1/2]^2,\\
  & 0 &  \quad  \text{otherwise}.\\
   \end{aligned}
               \right.  \\
\end{equation*}
For this problem, we consider $W$-cycles multigrid methods  with inner 3 $V(2,2)$-cycles  on the Schur complement system with $\alpha_B=1.5, \omega_B=\frac{3}{4}\alpha_B, \omega_J=0.8$.  In Tables \ref{tab:Iteration-number-h-beta-nu1} and   \ref{tab:Iteration-number-h-beta-nu2}, we report the smallest number of iterations $n$   such that $\frac{||d_h^{(n)}||_2}{||d_h^{(0)}||_2}<10^{-10}$ for this distributed control model problem for different values of $\beta$, with $\nu=1$ and $\nu=2$, respectively. Again, from Tables \ref{tab:Iteration-number-h-beta-nu1} and   \ref{tab:Iteration-number-h-beta-nu2}, we see the robustness of our multigrid methods  with respect to the regularization parameters, $\beta$, and grid size $h$.

 \begin{table}[H]
 \caption{Iteration counts for the control model problem with  state equation \eqref{eq:sub-state-equ-nonzero} using  $\nu=1$ and $W$-cycles multigrid for inexact stiffness-based BSR with inner 3 $V(2,2)$-cycles and    algorithmic parameters, $\alpha_B=1.5, \omega_B=\frac{3}{4}\alpha_B, \omega_J=0.8$. }
\centering
\begin{tabular}{ l    c c c c}
\hline
$h$      & $\beta=10^{-2}$  &$\beta=10^{-4}$    &$\beta=10^{-6}$   & $\beta=10^{-8}$  \\
\hline
 $1/64$                     &17          & 17     &   17       &  19    \\
 $1/128$                   &  17             &17        &  17        &   18      \\
 $1/256$                   &18        &18      &18       &18   \\
  $1/512$                 &19       &19      & 19   &19\\
\hline
\end{tabular}\label{tab:Iteration-number-h-beta-nu1}
\end{table}

 \begin{table}[H]
 \caption{Iteration counts for the control model problem with  state equation \eqref{eq:sub-state-equ-nonzero}  using $\nu=2$ and $W$-cycles multigrid for inexact stiffness-based BSR with inner 3 $V(2,2)$-cycles and   algorithmic parameters, $\alpha_B=1.5, \omega_B=\frac{3}{4}\alpha_B, \omega_J=0.8$. }
\centering
\begin{tabular}{ l  c c c c c}
\hline
$h$      & $\beta=10^{-2}$  &$\beta=10^{-4}$    &$\beta=10^{-6}$   & $\beta=10^{-8}$  \\
\hline
 $1/64$                     & 12      & 12          &  11        &  10     \\
 $1/128$                   & 12        &  12      & 12          &  11      \\
 $1/256$                   &13         &13         & 12          &12    \\
  $1/512$                 &13         &13      & 13    &13 \\
\hline
\end{tabular}\label{tab:Iteration-number-h-beta-nu2}
\end{table}
 
\section{Conclusion} \label{sec:concl}
In this work, we present a novel multigrid method for the elliptic distributed control problems,  where stiffness-based  Braess-Sarazin multigrid relaxation scheme is proposed to the discrete saddle-point system of the control problems.  In this relaxation scheme, the stiffness matrix obtained from the five-point finite difference method for the Laplacian is used to approximate the mass matrix in the saddle-point system discretized by a finite element method. We hire LFA to minimize LFA smoothing factor for Braess-Sarazin  relaxation, leading to  a highly satisfactory convergence factor.   For practical use,  we develop  an inexact version of Braess-Sarazin  relaxation,  where a 3 $V(2,2)$-cycles multigrid with weighted Jacobi relaxation is applied to the Schur complement system. This inexact version preserves the optimal smoothing factor observed for exact Braess-Sarazin  relaxation.  Numerical results show the robustness of our multigrid methods with respect to the regularization parameters and grid size.  In future, we will extend the relaxation scheme considered here to  more complex optimal problems, such as the Stokes control problem.

\bibliographystyle{siam}
\bibliography{mass_control_bib}
\end{document}